\documentclass[10pt]{amsart}
\usepackage[utf8]{inputenc} 
\usepackage[english,russian]{babel}
\usepackage{amsmath}
\usepackage{amssymb}
\usepackage{amsfonts}
\usepackage[dvips]{graphicx}
\usepackage{color}
\usepackage{tikz}
\usepackage[hidelinks]{hyperref}

\def\udcs{519.143} 
\def\mscs{05B15} 
\newtheorem{lemma}{Lemma}
\newtheorem{theorem}{Theorem}

\newtheorem{corollary}{Corollary}

\def\logo{{\bf\huge S\raisebox{0.2ex}{\hspace{0.55ex}\raisebox{0.05ex}e\hspace{-1.65ex}$\bigcirc$}MR}}

\def\semrtop
     {
  \vbox{
     \noindent\logo
     \hfill\raisebox{1ex}{ISSN 1813-3304}\par

     \vspace{5mm}

     \begin{center}
     {\huge SIBIRSKIE \ \`ELEKTRONNYE} \\[2mm]  
     {\huge MATEMATICHESKIE IZVESTIYA} \\[2mm]
     {\large Siberian Electronic Mathematical Reports} \\[1mm]
     {\LARGE\tt{http://semr.math.nsc.ru}}\\[0.5mm]
     \end{center}
     \vspace{-3mm}
     \noindent
     \begin{tabular}{c}
     \hphantom{aaaaaaaaaaaaaaaaaaaaaaaaaaaaaaaaaaaaaaaaaaaaaaaaaaaaaaaaaaaaaaaaaaaaaa} \\
     \hline\hline
     \end{tabular}

     \vspace{1mm}
     {\flushleft\it Vol 12, pp. 508--512 (2015) \hfill{\rm\small UDC \udcs}} 
     \newline
         {\rm\small DOI~\href{http://dx.doi.org/10.17377/semi.2015.12.043}{10.17377/semi.2015.12.043}}\hfill{\rm\small MSC\ \ \mscs }\par
  }
}

\begin{document}
\renewcommand{\refname}{References}
\renewcommand{\proofname}{Proof.}
\renewcommand\figurename{Figure}
\thispagestyle{empty}

\title[On the number of maximum independent sets in Doob graphs]{On the number of maximum independent sets\\ in Doob graphs}
\author{{D.\,S.\,Krotov}}%
\address{Denis Stanislavovich Krotov
\newline\hphantom{iii} Sobolev Institute of Mathematics,
\newline\hphantom{iii} pr. Akademika Koptyuga, 4,
\newline\hphantom{iii} 630090, Novosibirsk, Russia}%
\email{krotov@math.nsc.ru}%

\thanks{\sc Krotov, D. S.,
On the number of maximum independent sets in Doob graphs}
\thanks{\copyright \ 2015 Krotov D.S}
\thanks{The results were presented at the International Conference and PhD Summer School
``Groups and Graphs, Algorithms and Automata'' (August 09-15, 2015, Yekaterinburg, Russia).}
\thanks{The work was funded by the Russian Science Foundation (grant No 14-11-00555).}%
\thanks{\it Received  August, 4, 2015, published  September, 11,  2015.}%

\semrtop \vspace{1cm}
\maketitle {\small
\begin{quote}
\noindent{\sc Abstract. } 
  The Doob graph $D(m,n)$ is a distance-regular graph with the same parameters as the Hamming graph $H(2m+n,4)$.
  The maximum independent sets in the Doob graphs are analogs of the distance-$2$ MDS codes in the Hamming graphs.
  We prove that the logarithm of the number of the maximum independent sets in $D(m,n)$ grows as $2^{2m+n-1}(1+o(1))$.
  The main tool for the upper estimation is constructing an injective map from the class of maximum independent sets in $D(m,n)$
  to the class of distance-$2$ MDS codes in $H(2m+n,4)$.
\medskip

 \end{quote}
}

\section{Introduction}

The Cartesian product $D(m,n)\stackrel{\scriptscriptstyle \mathrm{def}}={\mathrm{Sh}}^m\times K_4^n$ 
of $m$ copies of the Shrikhande graph ${\mathrm{Sh}}$ 
(see Figure~\ref{fig:Sh0123}) 
\begin{figure}[hbt]
\begin{center}
\begin{tikzpicture}[
scale=0.8,
nn/.style={circle,fill=white,draw=black,thick, 
           inner sep=0.7pt}]
           \begin{scope} 
\clip [] (-0.45,-0.45) rectangle (3.45,3.45);
\draw[ystep=1,xstep=1,very thick] (-4.9,-2.1) grid (5.4,3.9);
\draw[xslant=1,ystep=9,xstep=1,very thick] (-3.4,-2.1) grid (6.4,3.9);
\end{scope}
\draw [dashed,thin] (-0.5,-0.5) rectangle (3.5,3.5);
\draw [dashed,thin,<->,rounded corners=3.5mm] (-0.5,2.7) -- (-1.2,2.7) -- (-1.2,3.8) -- (4.2,3.8) --(4.2,2.7) -- (3.5,2.7);
\draw [dashed,thin,<->,rounded corners=3.5mm] (-0.5,0.3) -- (-1.2,0.3) -- (-1.2,-0.8) -- (4.2,-0.8) --(4.2,0.3) -- (3.5,0.3);
\draw [dashed,thin,<->,rounded corners=3.5mm] (0.3,-0.5) -- (0.3,-1.2) -- (-0.8,-1.2) -- (-0.8,4.2) --(0.3,4.2) -- (0.3,3.5);
\draw [dashed,thin,<->,rounded corners=3.5mm] (2.7,-0.5) -- (2.7,-1.2) -- (3.8,-1.2) -- (3.8,4.2) --(2.7,4.2) -- (2.7,3.5);
\draw 
  (0,0) node [nn] {$ 00$} +(1,0) node [nn] {$ 10$} +(2,0) node [nn] {$ 20$} +(3,0) node [nn] {$ 30$}
++(0,1) node [nn] {$ 01$} +(1,0) node [nn] {$ 11$} +(2,0) node [nn] {$ 21$} +(3,0) node [nn] {$ 31$}
++(0,1) node [nn] {$ 02$} +(1,0) node [nn] {$ 12$} +(2,0) node [nn] {$ 22$} +(3,0) node [nn] {$ 32$}
++(0,1) node [nn] {$ 03$} +(1,0) node [nn] {$ 13$} +(2,0) node [nn] {$ 23$} +(3,0) node [nn] {$ 33$};
\end{tikzpicture}
\end{center}
\caption{The Shrikhande graph drown on a torus; the vertices are identified with the elements of $Z_4^2$}
\label{fig:Sh0123}
\end{figure}
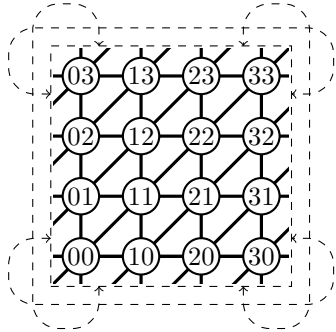
and $n$ copies of the complete graph $K_q$ of order $q=4$
is called a \emph{Doob graph} if $m>0$, while $D(0,n)$ is the \emph{Hamming graph} $H(n,4)$ 
(in general $H(n,q)\stackrel{\scriptscriptstyle \mathrm{def}}=K_q^n$).
The Doob graph $D(m,n)$ is a distance-regular graph with the same parameters as $H(2m+n,4)$, see e.g. \cite[\S9.2.B]{Brouwer}.
It is easy to see that the independence number of this graph is $4^{2m+n-1}$.
The maximum independent sets in the Hamming graphs are known as the distance-$2$ MDS codes 
(below, simply \emph{MDS codes}), or
the latin hypercubes 
(in the last case, 
one of the coordinates is usually considered as a function of the others).
It is naturally to call the maximum independent sets in Doob graphs by the same notion, 
the \emph{MDS codes}.
Indeed, the maximum independent sets in $D(m,n)$ and the MDS codes in $H(2m+n,4)$
have the same parameters being considered as error-correcting codes (see \cite{MWS} for the background on error-correcting codes)  
and as completely regular codes 
(see e.g. \cite[\S11.3]{Brouwer}).
(The concept of the latin hypercubes can also be generalized to $D(m,n)$; 
however, to do this,
we need at least one $K_4$ coordinate to treat as dependent, i.e., $n>0$.)
There are $4$ trivial MDS codes in $D(0,1)$; 
$24$ equivalent MDS codes in $D(0,2)$
($16$ of them can be found in Figure~\ref{fig:krotov1}); 
$16$ MDS codes in $D(1,0)$ 
(see Figure~\ref{fig:krotov1}), 
which form two equivalence classes (with $4$ and $12$ representatives, respectively).

The main result of the current correspondence is the following.

\begin{theorem}\label{th:}
The number of the maximum independent sets (distance-$2$ MDS codes) in the Doob graph $D(m,n)$
grows as $2^{2^{2m+n-1}(1+o(1))}$ as $(2m+n) \to \infty$.
\end{theorem}
The statement of the theorem is straightforward from Corollaries~\ref{c:ub} (an upper bound) and~\ref{c:lb} (a lower bound)
proven in the next two sections.

\section{An upper bound}

In this section, we describe a rather simple recursive way to map
injectively the set $\mathrm{MDS}_{m,n}$ of MDS codes in $D(m,n)$ into $\mathrm{MDS}_{0,2m+n}$. 
At first, we define the map $\xi$ from $\mathrm{MDS}_{1,0}$ into
$\mathrm{MDS}_{0,2}$, see Figure~\ref{fig:krotov1}. 

\def\shpart#1 #2 #3 #4!{node [#1] {} +(1,0) node [#2] {} +(2,0) node [#3] {} +(3,0) node [#4] {}}
\def\SS#1#2#3#4{\begin{tikzpicture}[
scale=0.3,
nz/.style={circle,fill=white,draw=black, 
           inner sep=1.4pt},
xz/.style={circle,fill=black!50!white,draw=black, 
           inner sep=1.4pt},
zz/.style={circle,fill=black,draw=black, 
           inner sep=1.4pt}]
\begin{scope} 
\clip [] (-0.4,-0.4) rectangle (3.4,3.4);
\draw[ystep=1,xstep=1] (-4.9,-2.1) grid (5.4,3.9);
\draw[xslant=1,ystep=9,xstep=1] (-3.4,-2.1) grid (6.4,3.9);
\end{scope}
\draw 
 (0,0) \shpart #4!
++(0,1) \shpart #3!
++(0,1) \shpart #2!
++(0,1) \shpart #1!;
\end{tikzpicture}}
\def\KK#1#2#3#4{\begin{tikzpicture}[
scale=0.3,
nz/.style={circle,fill=white,draw=black, 
           inner sep=1.4pt},
xz/.style={circle,fill=black!50!white,draw=black, 
           inner sep=1.4pt},
zz/.style={circle,fill=black,draw=black, 
           inner sep=1.4pt}]
\begin{scope} 
\clip [] (-0.05,-0.05) rectangle (3.05,3.05);
\draw[ystep=1,xstep=1] (-4.9,-2.1) grid (5.4,3.9);
\end{scope}
\draw[] (0,3) to [out=-110,in=110]  (0,0); \draw[] (0,2) to [out=-60,in=60]  (0,0); \draw[] (0,3) to [out=-60,in=60]  (0,1);
\draw[] (1,3) to [out=-110,in=110]  (1,0); \draw[] (1,2) to [out=-60,in=60]  (1,0); \draw[] (1,3) to [out=-60,in=60]  (1,1);
\draw[] (2,3) to [out=-110,in=110]  (2,0); \draw[] (2,2) to [out=-60,in=60]  (2,0); \draw[] (2,3) to [out=-60,in=60]  (2,1);
\draw[] (3,3) to [out=-110,in=110]  (3,0); \draw[] (3,2) to [out=-60,in=60]  (3,0); \draw[] (3,3) to [out=-60,in=60]  (3,1);
\draw[] (0,0) to [out=-20,in=-160]  (3,0); \draw[] (0,0) to [out=30,in=150]  (2,0); \draw[] (1,0) to [out=30,in=150]  (3,0);
\draw[] (0,1) to [out=-20,in=-160]  (3,1); \draw[] (0,1) to [out=30,in=150]  (2,1); \draw[] (1,1) to [out=30,in=150]  (3,1);
\draw[] (0,2) to [out=-20,in=-160]  (3,2); \draw[] (0,2) to [out=30,in=150]  (2,2); \draw[] (1,2) to [out=30,in=150]  (3,2);
\draw[] (0,3) to [out=-20,in=-160]  (3,3); \draw[] (0,3) to [out=30,in=150]  (2,3); \draw[] (1,3) to [out=30,in=150]  (3,3);
\draw 
 (0,0) \shpart #4!
++(0,1) \shpart #3!
++(0,1) \shpart #2!
++(0,1) \shpart #1!;
\end{tikzpicture}}
\begin{figure}[ht]
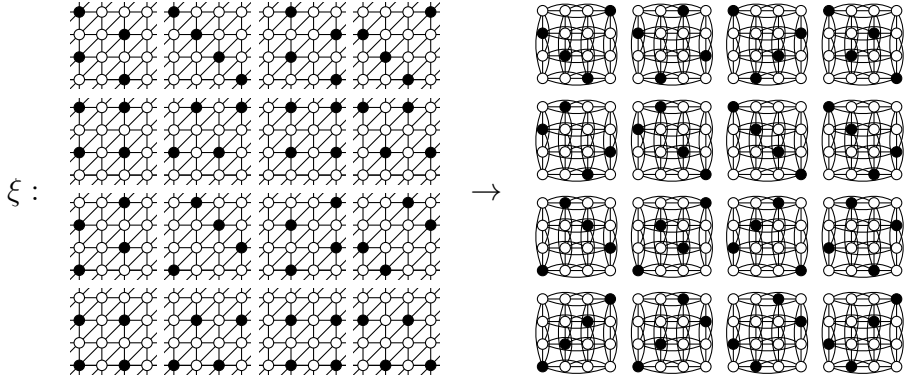

$$
\mbox{\Large$\xi:$}\ \ 
\begin{array}{c@{\ }c@{\ }c@{\ }c}
\SS{zz nz nz nz}{nz nz zz nz}{zz nz nz nz}{nz nz zz nz}&
\SS{zz nz nz nz}{nz zz nz nz}{nz nz zz nz}{nz nz nz zz}&
\SS{nz zz nz nz}{nz nz nz zz}{nz zz nz nz}{nz nz nz zz}&
\SS{nz nz nz zz}{zz nz nz nz}{nz zz nz nz}{nz nz zz nz}
\\
\SS{zz nz zz nz}{nz nz nz nz}{zz nz zz nz}{nz nz nz nz}&
\SS{nz zz nz zz}{nz nz nz nz}{zz nz zz nz}{nz nz nz nz}&
\SS{nz zz nz zz}{nz nz nz nz}{nz zz nz zz}{nz nz nz nz}&
\SS{zz nz zz nz}{nz nz nz nz}{nz zz nz zz}{nz nz nz nz}
\\
\SS{nz nz zz nz}{zz nz nz nz}{nz nz zz nz}{zz nz nz nz}&
\SS{nz zz nz nz}{nz nz zz nz}{nz nz nz zz}{zz nz nz nz}&
\SS{nz nz nz zz}{nz zz nz nz}{nz nz nz zz}{nz zz nz nz}&
\SS{nz nz zz nz}{nz nz nz zz}{zz nz nz nz}{nz zz nz nz}
\\
\SS{nz nz nz nz}{zz nz zz nz}{nz nz nz nz}{zz nz zz nz}&
\SS{nz nz nz nz}{nz zz nz zz}{nz nz nz nz}{zz nz zz nz}&
\SS{nz nz nz nz}{nz zz nz zz}{nz nz nz nz}{nz zz nz zz}&
\SS{nz nz nz nz}{zz nz zz nz}{nz nz nz nz}{nz zz nz zz}
\end{array}
\ \  \mbox{\Large$\to$}\ \ %
\begin{array}{c@{\ }c@{\ }c@{\ }c}
\KK{nz nz nz zz}{zz nz nz nz}{nz zz nz nz}{nz nz zz nz}&
\KK{nz nz zz nz}{zz nz nz nz}{nz nz nz zz}{nz zz nz nz}&
\KK{zz nz nz nz}{nz nz nz zz}{nz nz zz nz}{nz zz nz nz}&
\KK{zz nz nz nz}{nz nz zz nz}{nz zz nz nz}{nz nz nz zz}
\\
\KK{nz zz nz nz}{zz nz nz nz}{nz nz nz zz}{nz nz zz nz}&
\KK{nz zz nz nz}{zz nz nz nz}{nz nz zz nz}{nz nz nz zz}&
\KK{zz nz nz nz}{nz zz nz nz}{nz nz zz nz}{nz nz nz zz}&
\KK{zz nz nz nz}{nz zz nz nz}{nz nz nz zz}{nz nz zz nz}
\\
\KK{nz zz nz nz}{nz nz zz nz}{nz nz nz zz}{zz nz nz nz}&
\KK{nz nz nz zz}{nz zz nz nz}{nz nz zz nz}{zz nz nz nz}&
\KK{nz nz zz nz}{nz zz nz nz}{zz nz nz nz}{nz nz nz zz}&
\KK{nz zz nz nz}{nz nz nz zz}{zz nz nz nz}{nz nz zz nz}
\\
\KK{nz nz nz zz}{nz nz zz nz}{nz zz nz nz}{zz nz nz nz}&
\KK{nz nz zz nz}{nz nz nz zz}{nz zz nz nz}{zz nz nz nz}&
\KK{nz nz zz nz}{nz nz nz zz}{zz nz nz nz}{nz zz nz nz}&
\KK{nz nz nz zz}{nz nz zz nz}{zz nz nz nz}{nz zz nz nz}
\end{array}
$$
\caption{The $16$ MDS codes in ${\mathrm{Sh}}$ and the corresponding MDS~codes~in~$K_4^2$}
\label{fig:krotov1}
\end{figure}

For arbitrary $m$, $n\ge 0$, the action of $\kappa: \mathrm{MDS}_{m+1,n}\to \mathrm{MDS}_{m,n+2}$ is defined as follows:
$$
\kappa M \stackrel{\scriptscriptstyle \mathrm{def}}= \big\{ (x_1,...,x_m,z_1,z_2,y_1,...,y_n)\in D(m,n+2) \,\big|\, (z_1,z_2)\in \xi M_{x_1,...,x_m,y_1,...,y_n}\big\},
$$ 
where
$$
M_{x_1,...,x_m,y_1,...,y_n} \stackrel{\scriptscriptstyle \mathrm{def}}= \{ v \in {\mathrm{Sh}} \mid  (x_1,...,x_m,v,y_1,...,y_n)\in M \}  
$$

\begin{lemma}\label{l:inj}
 For every MDS code in $D(m+1,n)$, the set $\kappa M$ is an MDS code in $D(m,n+2)$.
\end{lemma}
\begin{proof}
 The map $\xi$ has the following important property,
which can be checked directly, see Figure~\ref{fig:krotov1}:
two MDS codes $M'$ and $M''$ in $D(1,0)$ intersect if and only if their images 
$\xi M'$ and $\xi M''$ intersect.
Since $M$ is an independent set,
 $M_{x_1,...,x_m,y_1,...,y_n}$ and $M_{u_1,...,u_m,w_1,...,w_n}$ 
(and hence, also $\xi M_{x_1,...,x_m,y_1,...,y_n}$ and $\xi M_{u_1,...,u_m,w_1,...,w_n}$)
are disjoint for any two neighbor vertices 
$(x_1,...,x_m,y_1,...,y_n)$ and $(u_1,...,u_m,w_1,\linebreak[2]...,\linebreak[2]w_n)$ of $D(m,n)$. 
It follows that $\kappa M$ 
is also an independent set.
Moreover, it has the same cardinality as $M$,
 i.e., $4^{2m+n+1}$.
\end{proof}

Then, $\kappa^m$, the $m$th iteration of $\kappa$, maps $\mathrm{MDS}_{m,n}$
 into $\mathrm{MDS}_{0,2m+n}$.

\begin{corollary}\label{c:ub}
 The number of the MDS codes in $D(m,n)$ 
 does not exceed
 $$2^{2^{2m+n-1}(1+o(1))}.$$
\end{corollary}
\begin{proof}
 Since $\kappa$ obviously maps different MDS codes to different MDS codes,
 the statement of the corollary in the general case 
 can be inductively reduced to the partial case $m=0$,
 which was proven in \cite{PotKro:asymp}, see also \cite{KroPot:Lyap}.
\end{proof}

\section{A lower bound}
In this section, we consider a simple way to construct doubly exponential
(with respect to the graph diameter $2m+n$)
number of MDS codes in the Doob graph $D(m,n)$.

The vertices of ${\mathrm{Sh}}$ 
will be identified with the pairs $ab$ 
(considered as a short notation for 
$(a,b)$
),
where $a$, $b\in \{0,1,2,3\}$, 
see Figure~\ref{fig:Sh0123}.
The vertices of $K_4$ will be identified
 with the pairs $ab$,
 where $a$, $b\in \{0,1\}$.
For every function $\lambda$ 
from $\{0,1,2,3\}^m\times\{0,1\}^n$
 to $\{0,1\}$,
 we define the set

\begin{eqnarray*}
 M_\lambda 
 \stackrel{\scriptscriptstyle \mathrm{def}}=
 \bigg\{(x'_1 x''_1,\ldots,x'_{m+n}x''_{m+n})\in D(m,n) &\Big|& \sum_{i=1}^{m+n} x'_i \equiv 0\bmod 2,\\
           &&\sum_{i=1}^{m+n} x''_i \equiv \lambda(x'_1,\ldots,x'_{m+n}) \bmod 2 \bigg\} .
\end{eqnarray*}

\begin{lemma}\label{l:22n}
  For any function
 $\lambda:\{0,1,2,3\}^m\times\{0,1\}^n\to\{0,1\}$, 
the set $M_\lambda$ 
is an MDS code in $D(m,n)$.
\end{lemma}
\begin{proof}
 It is easy to see that if $m< i\le m+n$, 
then for any values of 
 $x'_1x''_1$, \ldots, $x'_{i-1}x''_{i-1}$, $x'_{i+1}x''_{i+1}$, \ldots, $x'_{m+n}x''_{m+n}$
 there is a unique pair $x'_ix''_i$ 
such that
 $(x'_1 x''_1,\ldots,x'_{m+n}x''_{m+n})\in M_\lambda$.
 If $1\le i\le m$, then there are four such pairs
(two possibilities for $x'_i$, of the same parity,
 and for each choice of $x'_i$, two possibilities for $x''_i$), 
 but they correspond to pairwise independent
 vertices of the Shrikhande graph.
 Consequently, at first,
 $|M_\lambda|=4^{2m+n-1}$,
 and at second, 
$M_\lambda$ is an independent set.
\end{proof}
\begin{corollary}\label{c:lb}
  There are at least $2^{2^{2m+n-1}}$
 different MDS codes in $D(m,n)$.
\end{corollary}
\begin{proof}
  We will say that two functions
 from $\{0,1,2,3\}^m\times\{0,1\}^n$ 
to $\{0,1\}$ 
  are essentially different
 if their values are different 
in at least one point
  $(x'_1,\ldots,x'_{m+n})$ 
satisfying 
$x'_1+\ldots+x'_{m+n} \equiv 0 \bmod 2$. 
  The number of essentially different functions is
 $2^{2^{2m+n-1}}$.
  Obviously, essentially different functions
 $\lambda$ 
lead to different MDS 
codes $M_\lambda$.
\end{proof}

\section{Conclusion}

We have established the asymptotics of $\log |\mathrm{MDS}_{m,n}|$, generalizing the similar result for the MDS codes 
in the Hamming graph $H(n,4)$ \cite{KroPot:Lyap}, \cite{PotKro:asymp}. 
Note that the case $q=4$ is the only nontrivial case when 
the asymptotics of the double logarithm of the number of MDS codes is known ($n\to\infty$, $q$ is fixed).
Known bounds for the other cases can be found in \cite{KPS:ir}, \cite{PotKro:numberQua};
the exact values for small $q$ and $n$, in \cite{MK-W:11}, \cite{MK-W:small}, \cite{PotKro:numberQua}.

A constructive characterization of the class $\mathrm{MDS}_{0,n}$
can be found in \cite{KroPot:4}.
A possibility to relate the MDS codes (maximum independent sets) in $D(m,n)$ with MDS codes in $D(2m+n,4)$ using the map $\kappa^m$ 
suggests that a similar characterization might be possible for $\mathrm{MDS}_{m,n}$ with arbitrary $m$.
However, it is not completely clear if the map $\kappa^m$ itself can be helpful for a reasonable proof of such characterization.
Since the map $\kappa$ is not point-to-point, the result of the $m$th iteration of $\kappa$ can depend on the order of coordinates.
As a result, it is not easy to track which subclass of $\mathrm{MDS}_{0,2m+n}$ we obtain as the image of $\mathrm{MDS}_{m,n}$ under $\kappa^m$
and to describe this subclass in terms of the known  characterization of $\mathrm{MDS}_{0,2m+n}$.
In any case, finding a characterization of the class $\mathrm{MDS}_{m,n}$, using $\kappa$ or not, will be a natural continuation of the current research.

\bigskip

\parbox[t][0cm]{12cm}{\mbox{}\\[3cm]
  \mbox{}\hrulefill\mbox{}\hfill\mbox{}\\
  \small
  {\sc $\vphantom{0}^\star$Further reading:}\\[1mm]
   [10] \parbox[t]{105mm}{D.~S. Krotov, E.~A. Bespalov.\\
 {\it Distance-2 MDS codes and latin colorings in the Doob graphs},\\
  \href{http://arxiv.org/abs/1510.01429}{arXiv:1510.01429} }
  }
\end{document}